\documentclass[letterpaper,11pt,twoside,keywordsasfootnote,addressatend,noinfoline]{article}
\usepackage{fullpage}
\usepackage[english]{babel}
\usepackage{amssymb}
\usepackage{amsmath}
\usepackage{theorem}
\usepackage{epsfig}
\usepackage{subfigure}
\usepackage{imsart}
\usepackage{color}

\theorembodyfont{\normalfont}
\newtheorem{theorem}{Theorem}

\newtheorem{lemma}[theorem]{Lemma}

\newenvironment{proof}{\noindent{\scshape Proof.}}{\hspace*{2mm} $\square$}

\newcommand{\Z}{\mathbb{Z}}
\newcommand{\R}{\mathbb{R}}

\newcommand{\ind}{\mathbf{1}}
\newcommand{\ep}{\epsilon}

\DeclareMathOperator{\card}{card}
\DeclareMathOperator{\bernoulli}{Bernoulli \,}
\DeclareMathOperator{\binomial}{Binomial \,}
\DeclareMathOperator{\uniform}{Uniform \,}

\DeclareMathOperator{\cont}{cont \,}
\DeclareMathOperator{\ax}{ax}

\begin{document}
\begin{frontmatter}
\title     {Clustering and coexistence in the one-dimensional \\ vectorial Deffuant model}
\runtitle  {One-dimensional vectorial Deffuant model}
\author    {Nicolas Lanchier\thanks{Research partially supported by NSF Grant DMS-10-05282} and
            Stylianos Scarlatos} 
\runauthor {N. Lanchier and S. Scarlatos}
\address   {School of Mathematical and Statistical Sciences \\ Arizona State University \\ Tempe, AZ 85287, USA.}
\address   {Department of Mathematics \\ University of Patras \\ Patras 26500, Greece.}

\maketitle

\begin{abstract} \ \
 The vectorial Deffuant model is a simple stochastic process for the dynamics of opinions that also includes a confidence threshold.
 To understand the role of space in this type of social interactions, we study the process on the one-dimensional lattice where
 individuals are characterized by their opinion -- in favor or against -- about~$F$ different issues and where pairs of nearest
 neighbors potentially interact at rate one.
 Potential interactions indeed occur when the number of issues both neighbors disagree on does not exceed a certain confidence threshold,
 which results in one of the two neighbors updating her opinion on one of the issues both neighbors disagree on (if any).
 This paper gives sufficient conditions for clustering of the system and for coexistence due to fixation in a fragmented configuration,
 showing the existence of a phase transition between both regimes at a critical confidence threshold.
\end{abstract}



\end{frontmatter}


\section{Introduction}\label{sec:introduction}

\indent In the voter model~\cite{clifford_sudbury_1973, holley_liggett_1975}, individuals are located on the vertex set of a graph
 and are characterized by one of two competing opinions.
 Individuals update their opinion independently at rate one by mimicking a random neighbor where the neighborhood is defined in an
 obvious manner from the edge set of the graph.
 This is the simplest model of opinion dynamics based on the framework of interacting particle systems.
 The model includes social influence, the tendency of individuals to become more similar when they interact.
 More recently, and particularly since the work of political scientist Axelrod~\cite{axelrod_1997}, a number of variants of the
 voter model that also account for homophily, the tendency of individuals to interact more frequently with individuals who
 are more similar, have been introduced.
 These spatial processes are continuous-time Markov chains whose state at time $t$ is a function that maps the vertex set $V$ of a graph
 into a set of opinions:
 $$ \eta_t : V \ \longrightarrow \ \Gamma \ := \ \hbox{opinion set}. $$
 The common modeling approach is to equip~$\Gamma$ with a metric, which allows to define an opinion distance between neighbors,
 and to include homophily by assuming that neighbors interact at a rate which is a nonincreasing function of their opinion distance.
 This rate is often chosen to be the step function equal to zero if the opinion distance between the two neighbors is larger
 than a so-called confidence threshold and equal to one otherwise. \vspace*{8pt}

\noindent{\bf Model description} --
 In the original version of the Deffuant model \cite{deffuant_al_2000}, the opinion space is the unit interval equipped with
 the Euclidean distance.
 Neighbors interact at rate one if and only if the distance between their opinion does not exceed a certain confidence threshold,
 which results in a compromise strategy where both opinions get closer to each other by a fixed factor.
 The main conjecture about this opinion model is that, at least when the initial opinions are independent and uniformly distributed
 over the unit interval, the system reaches a consensus when the confidence threshold is larger than one half whereas disagreements
 persist in the long run when the confidence threshold is smaller than one half.
 This conjecture has been completely proved for the one-dimensional system in~\cite{haggstrom_2012, lanchier_2012b} using different
 techniques and we also refer to~\cite{haggstrom_hirscher_2014} for additional results on the system in higher dimensions and/or
 starting from more general distributions.
 In contrast, in the vectorial version of the Deffuant model also introduced in~\cite{deffuant_al_2000}, the opinion space is the
 hypercube equipped with the Hamming distance:
 $$ \Gamma \ := \ \{0, 1 \}^F \quad \hbox{and} \quad H (u, v) \ := \ \card \,\{i : u_i \neq v_i \} \quad \hbox{for all} \quad u, v \in \Gamma. $$
 As for the general class of opinion models described above, the system depends on a confidence threshold that we call~$\theta$
 from now on.
 To describe the dynamics, we also introduce
 $$ \Omega (x, y, \eta) \ := \ \{u \in \Gamma : H (u, \eta (y)) = H (\eta (x), \eta (y)) - 1 \} $$
 for each pair of neighbors $x$ and $y$ and each configuration $\eta$.
 The vectorial Deffuant model can then be formally defined as the continuous-time Markov chain with generator
\begin{equation}
 \label{eq:model}
 \begin{array}{rcl}
   Lf (\eta) & = & \sum_x \,(\card \,\{y : y \sim x \})^{-1} \,\sum_{y \sim x} \,(\card \,\Omega (x, y, \eta))^{-1} \vspace*{4pt} \\
             &   & \hspace*{40pt} \sum_{u \in \Omega (x, y, \eta)} \,\ind \{1 \leq H (\eta (x), \eta (y)) \leq \theta \} \ [f (\eta_{x, u}) - f (\eta)] \end{array}
\end{equation}
 where $y \sim x$ means that both vertices are nearest neighbors and where~$\eta_{x, u}$ is the configuration obtained from
 configuration~$\eta$ by setting the opinion at~$x$ equal to~$u$ and leaving all the other opinions unchanged.
 In words, each individual looks at a random neighbor at rate one and updates her opinion by moving one unit towards the opinion
 of this neighbor along a random direction in the hypercube unless either the opinion distance between the two neighbors exceeds
 the confidence threshold or both neighbors already agree.
 These evolution rules, which are somewhat complicated thinking of each opinion as an element of the hypercube, have a very natural
 interpretation if one thinks of each opinion as a set of binary opinions -- in favor or against -- about $F$ different issues.
 Using this point of view gives the following:
 each individual looks at a random neighbor at rate one and imitates the opinion of this neighbor on an issue selected uniformly at
 random among the issues they disagree on (if any), which models social influence, unless the number of issues they disagree on
 exceeds the confidence threshold, which models homophily. \vspace*{8pt}

\noindent{\bf Main results} --
 The results in~\cite[section 4]{deffuant_al_2000} are based on numerical simulations of the system on a complete graph where
 all the individuals are neighbors of each other, thus leaving out any spatial structure.
 In contrast, the main objective of this paper is to understand not only the role of the parameters but also the role of explicit
 space in the long-term behavior of the system.
 In particular, we specialize from now on in the one-dimensional system where each individual has exactly two nearest neighbors
 and set $V = \Z$.
 In this case, the process can exhibit two types of behavior:
 reach a consensus or get trapped in an absorbing state where the different opinions coexist.
 To define mathematically this dichotomy, we say that
\begin{itemize}
 \item the system {\bf clusters} whenever $\lim_{t \to \infty} P \,(\eta_t (x) = \eta_t (y)) = 1$ for all $x, y \in \Z$, \vspace*{2pt}
 \item the system {\bf coexists} due to fixation whenever
  $$ P \,(\eta_t (x) = \eta_{\infty} (x) \ \hbox{eventually in $t$}) \ = \ 1 \quad \hbox{for all} \quad x \in \Z $$
  for some configuration $\eta_{\infty}$ such that $P \,(\card \,\{\eta_{\infty} (x) : x \in \Z \} = 2^F) = 1$.
\end{itemize}
 First, we note that, when $F \leq \theta$, the process reduces to a superposition of $F$ voter models:
 the configuration of opinions for each given issue evolves according to a voter model which is slowed down by a factor equal
 to the number of issues on which neighbors disagree.
 In particular, it directly follows from~\cite{clifford_sudbury_1973, holley_liggett_1975} that the system clusters.
 In the nontrivial case $F > \theta$, the system has been studied numerically by the authors of~\cite{adamopoulos_scarlatos_2012}
 who considered a percolation model starting from the uniform product measure and predicted a phase transition
 for the continuous-time model between the regimes of clustering and coexistence at an approximate critical
 value $\theta_c \approx F/2$.
 To begin with, we follow~\cite{adamopoulos_scarlatos_2012} and study the one-dimensional system starting from the uniform product
 measure in which the opinions at different vertices are independent and equally likely, i.e.,
\begin{equation}
\label{eq:uniform}
  P \,(\eta_0 (x) = u) \ = \ (1/2)^F \quad \hbox{for all} \quad x \in \Z \ \ \hbox{and} \ \ u \in \Gamma.
\end{equation}
 The first key ingredient to prove both clustering and coexistence is to think of each opinion profile as a collection of~$F$~levels
 each having two possible states and put a particle between two neighbors at the levels they disagree on.
 This induces a coupling between the dynamics of opinions and a system of annihilating random walks similar to the one introduced
 in~\cite{lanchier_schweinsberg_2012}.
 The inclusion of a confidence threshold in the opinion model translates into the following in the system of random walks:
 particles jump at a positive rate except the ones that are part of a pile whose size exceeds the confidence threshold which do not
 move because they are carried by an edge connecting two individuals who disagree too much to interact.
 We call a particle either active or frozen depending on whether it jumps at a positive rate or cannot jump at all.
 When only individuals who disagree on all issues cannot interact, only piles of exactly $F$ particles are frozen and the machinery
 in~\cite{lanchier_schweinsberg_2012}, which partly relies on a delicate symmetry argument due to Adelman~\cite{adelman_1976},
 easily extends to prove that each frozen particle will, after an almost surely finite time, either become active or annihilate
 with an active particle.
 From this, it can be deduced that both frozen and active particles ultimately go extinct, which is equivalent to clustering
 therefore we have the following result.
\begin{theorem} --
\label{th:flux}
 Assume~\eqref{eq:uniform} and $F = \theta + 1$. Then, the system clusters.
\end{theorem}
\begin{figure}[t]
\centering
\scalebox{0.32}{\input{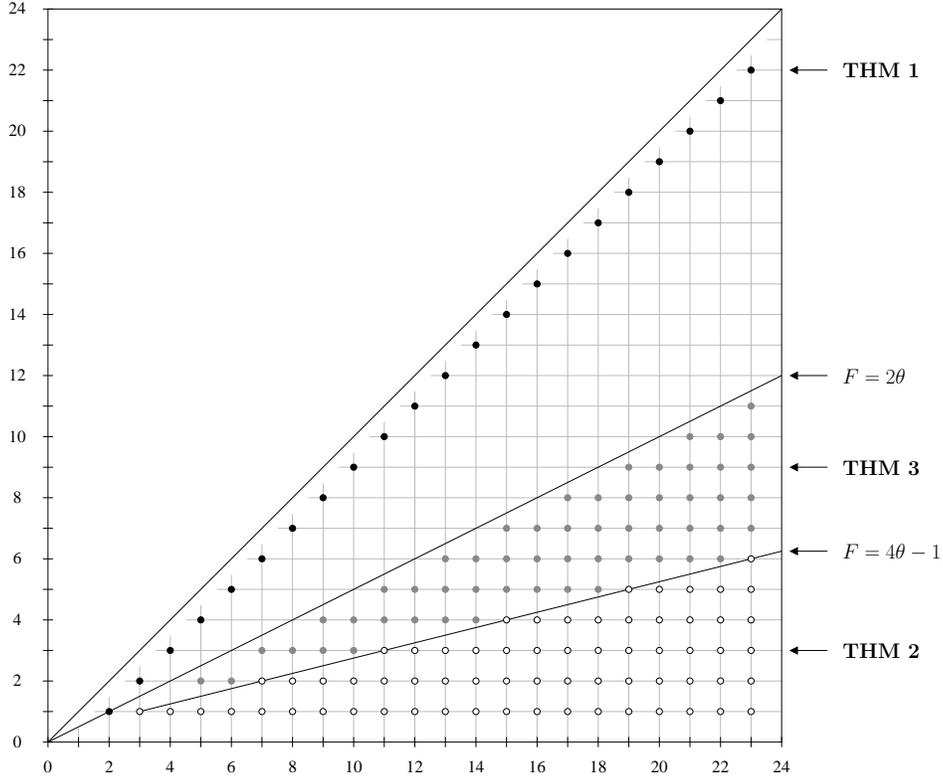}}
\caption{\upshape{Phase diagram of the one-dimensional vectorial Deffuant model in the $F - \theta$ plane along with a summary
 of our theorems.
 The black dots correspond to the set of parameters for which clustering is proved whereas the white dots
 correspond to the set of parameters for which coexistence is proved.}}
\label{fig:diagram}
\end{figure}
 To study the coexistence regime, we again use the coupling with annihilating random walks as well as a characterization of fixation
 based on certain spatial properties of so-called active paths that keep track of the offspring of the binary opinions initially
 present in the system.
 This characterization leads to a sufficient condition for survival of the frozen particles on a large interval, and therefore
 coexistence due to fixation, based on the initial number of active and frozen particles in this interval.
 Using estimates on the random number of active particles that annihilate with frozen particles to turn a pile of frozen particles
 into a smaller pile of active particles, a random weight is then attributed to each pile of particles at time zero.
 This, together with large deviation estimates for the cumulative weight in large intervals, implies that the system coexists
 whenever the expected value of the weight of a typical pile is positive.
 Relying finally on some symmetry property of the binomial random variable allows to make explicit the set of parameters for which
 the expected value of the weight is positive, from which we deduce the following theorem.
\begin{theorem} --
\label{th:uniform-fix}
 Assume~\eqref{eq:uniform} and $F \geq 4 \theta - 1$. Then, the system coexists.
\end{theorem}
 Note that the combination of both theorems implies the existence of at least one phase transition between consensus and coexistence
 at some critical confidence threshold
 $$ \theta_c \in ((1/4)(F + 1), F - 1) \quad \hbox{for all} \quad F \geq 2 $$
 which gives a rigorous proof of part of the conjecture in~\cite{adamopoulos_scarlatos_2012}.
 To gain some insight on the reason why the critical threshold might indeed be $F/2$, we finally look at the system starting from
 a non-uniform product measure where two opposite designated opinion profiles start at high density whereas the other opinion profiles
 start at low density.
 More precisely, we assume that
\begin{equation}
\label{eq:biased}
  \begin{array}{rclcl}
     P \,(\eta_0 (x) = u) & = & 1/2 - (2^{F - 1} - 1) \,\rho & \hbox{when} & u \in \{u_-, u_+ \} \vspace*{2pt} \\
                          & = & \rho & \hbox{when} & u \notin \{u_-, u_+ \} \end{array}
\end{equation}
 where $\rho \in [0, 2^{-F})$ is a small parameter and where
 $$ u_- \ := \ (0, 0, \ldots, 0) \quad \hbox{and} \quad u_+ := (1, 1, \ldots, 1). $$
 For the process starting from such an initial distribution, the methodology developed to proved Theorem~\ref{th:uniform-fix} can again
 be applied which, together with large deviation estimates for non-independent random variables, gives the following result.
\begin{theorem} --
\label{th:biased-fix}
 Assume~\eqref{eq:biased} and $F > 2 \theta$.
 Then, the system coexists for all $\rho > 0$ small.
\end{theorem}
 Even though the theorem only gives a sufficient condition for coexistence due to fixation, the proof somewhat suggests that this
 condition is also necessary.
 Our intuition relies on the fact that the largest blockades contain~$F$ frozen particles while
 the collision of $F - \theta$ active particles with such a blockade can create a total of $\theta$ active particles.
 In particular, when $F \leq 2 \theta$, it is possible that the number of active particles created is at least equal to the
 number of active particles destroyed, which leads ultimately to a global extinction of all the particles and therefore clustering.
 We refer the reader to Figure~\ref{fig:diagram} for a summary of our results.
 The rest of this paper is devoted to proofs starting in the next section with the coupling with annihilating random walks which is
 then used to show our three theorems in the subsequent three sections.


\section{Coupling with annihilating random walks}
\label{sec:coupling}

\noindent In this section, we follow the approach of \cite{lanchier_schweinsberg_2012} to define a coupling between the process
 and a collection of systems of symmetric annihilating random walks.
 The basic idea is to visualize each opinion profile, i.e., each vertex of the hypercube, using~$F$~levels each having two
 possible states and put particles between two neighbors at the levels they disagree on.
 For an illustration, we refer the reader to Figure~\ref{fig:layers} where black and white dots represent the two possible opinions
 on each issue and where the crosses indicate the position of the particles.
 To make this construction rigorous, we first identify the process with the spin system
 $$ \bar \eta_t : \Z \times \{1, 2, \ldots, F \} \ \to \ \{0, 1 \} \quad \hbox{where} \quad \bar \eta_t (x, i) := \hbox{$i$th coordinate of} \ \eta_t (x). $$
 This again defines a Markov process.
 To describe this system of particles, it is also convenient to identify the edges connecting neighbors with their midpoint
 $$ e \ := \ (x, x + 1) \ \equiv \ x + 1/2 \quad \hbox{for all} \quad x \in \Z $$
 and to define translations on this set of edges by setting
 $$ e + a \ := \ (x, x + 1) + a \ \equiv \ x + 1/2 + a \quad \hbox{for all} \quad e \in \Z + 1/2 \quad \hbox{and all} \quad a \in \R. $$
 The process that keeps track of the disagreements is then defined as another spin system
\begin{equation}
\label{eq:disagreement}
  \xi_t (e, i) \ := \ \ind \,\{\bar \eta_t (e - 1/2, \,i) \neq \bar \eta_t (e + 1/2, \,i) \}
\end{equation}
 and we put a particle on edge $e$ at level $i$ if and only if $\xi_t (e, i) = 1$ to visualize the corresponding configuration
 of interfaces.
 The reason for introducing this system of particles is that
\begin{itemize}
\item the limiting behavior of the process~\eqref{eq:model} can be easily translated into simple properties for the system of particles~\eqref{eq:disagreement}:
 clustering is equivalent to extinction of the particles whereas coexistence is equivalent to survival of the particles and \vspace*{4pt}
\item the particles of~\eqref{eq:disagreement} consist of a collection of systems of simple symmetric annihilating random walks somewhat easier to analyze than
 the vectorial Deffuant model itself.
\end{itemize}
 The number of particles per edge, defined as
 $$ \begin{array}{l} \zeta_t (e) \ := \ \xi_t (e, 1) + \xi_t (e, 2) + \cdots + \xi_t (e, F) \quad \hbox{for each edge} \ e, \end{array} $$
 is a key quantity to fully describe the dynamics of the system of particles since it is equal to the opinion distance between the two
 individuals connected by the edge.
 Indeed,
 $$ \begin{array}{rcl}
    \zeta_t (e) & = & \card \,\{i : \xi_t (e, i) = 1 \} \vspace*{4pt} \\
                & = & \card \,\{i : \bar \eta_t (e - 1/2, \,i) \neq \bar \eta_t (e + 1/2, \,i) \} \vspace*{4pt} \\
                & = &  H (\eta_t (e - 1/2), \eta_t (e + 1/2)). \end{array} $$
 Using that the number of particles on the edge is equal to the opinion distance, we deduce that an interaction along edge $e := (x, x + 1)$
 results in the following alternative:
\begin{enumerate}
 \item There are more than $\theta$ particles on the edge in which case nothing happens because the opinion distance between the
   neighbors exceeds the confidence threshold. \vspace*{4pt}
 \item There are at most $\theta$ particles on the edge in which case one of the issues for which the two neighbors disagree is chosen
   uniformly at random and the opinion of either vertex~$x$ or vertex~$x + 1$ at this level is switched.
   Note that the issues the two neighbors disagree on correspond to the levels which are occupied by a particle so, after the
   interaction, the particle at the chosen level disappears while the state of one of the two edges $e \pm 1$ at the same level
   switches from either empty to occupied or from occupied to empty.
\end{enumerate}
 Combining 1 and 2, we deduce that the system of particles~\eqref{eq:disagreement} evolves at each level according to a system of simple
 symmetric annihilating random walks as illustrated in Figure~\ref{fig:layers}.
 In addition, since the issues on which neighbors disagree are chosen for update uniformly at random, at each edge occupied by
 $j$~particles, these particles jump individually at rate
\begin{equation}
\label{eq:rate-deffuant}
  \begin{array}{rclcl}
   r (j) & = & j^{-1} & \hbox{when} & 0 < j \leq \theta \vspace*{4pt} \\ & = &   0 & \hbox{when} & \theta < j \leq F \end{array}
\end{equation}
 making the~$F$ systems of symmetric annihilating random walks non-independent.
 Motivated by the transition rates in \eqref{eq:rate-deffuant}, we call an edge either a {\bf live edge} or a {\bf blockade} depending on
 whether they have at most or more than $\theta$ particles, respectively.
 Accordingly, we call the particles at this edge either {\bf active} or {\bf frozen} particles, respectively, and notice
 that active particles jump at a positive rate whereas frozen particles cannot jump at all.

\begin{figure}[t]
\centering
\scalebox{0.42}{\input{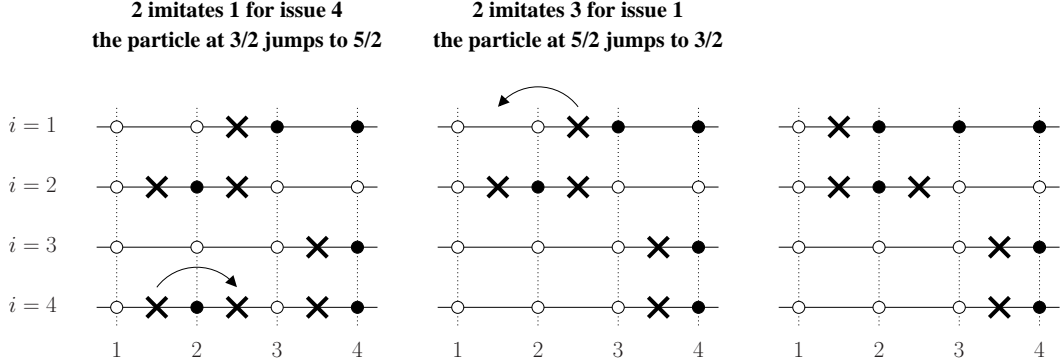}}
\caption{\upshape{Illustration of the coupling between the vectorial Deffuant model and the system of simple symmetric annihilating
 random walks.
 Black and white dots represent the two possible states of the individuals' opinion on each issue while the crosses indicate the
 position of the particles.
 In our example, there are $2^4 = 16$ possible opinion profiles and the confidence threshold is $\theta \geq 2$.
 The two imitation events represented in this realization translate into two consecutive jumps of particles, with the first one
 resulting in the annihilation of two particles.}}
\label{fig:layers}
\end{figure}

\indent Both the vectorial Deffuant model and its coupled systems of annihilating random walks starting from any initial configuration can
 be constructed from the same percolation structure using a standard argument due to Harris \cite{harris_1972}.
 This percolation structure consists of a random graph involving independent Poisson processes marking the times at which potential jumps
 or interactions occur and additional collections of independent Bernoulli random variables and uniform random variables to determine the
 outcome of each jump or interaction.
 To make this construction rigorous, for each pair of individual-issue or vertex-level $(x, i) \in \Z \times \{1, 2, \ldots, F \}$,
\begin{itemize}
 \item we let $(N_{x, i} (t) : t \geq 0)$ be a rate one Poisson process, \vspace*{4pt}
 \item we denote by $T_{x, i} (n)$ its $n$th arrival time: $T_{x, i} (n) := \inf \,\{t : N_{x, i} (t) = n \}$, \vspace*{4pt}
 \item we let $(B_{x, i} (n) : n \geq 1)$ be a collection of independent Bernoulli variables with
  $$ P \,(B_{x, i} (n) = + 1) \ = \ P \,(B_{x, i} (n) = - 1) \ = \ 1/2, $$
 \item and we let $(U_{x, i} (n) : n \geq 1)$ be a collection of independent $\uniform (0, 1)$.
\end{itemize}
 Then, at each time $t := T_{x, i} (n)$, we draw an arrow
\begin{equation}
\label{eq:arrow}
  x \ \to \ y := x + B_{x, i} (n) \quad \hbox{with the label~$i$}
\end{equation}
 and call this arrow an {\bf active} arrow if and only if
\begin{equation}
\label{eq:active}
  \xi_{t-} (e, i) \ = \ 1 \quad  \hbox{and} \quad U_{x, i} (n) \ \leq \ r (\zeta_{t-} (e)) \quad \hbox{where} \quad e \ := \ x + (1/2) \,B_{x, i} (n).
\end{equation}
 The vectorial Deffuant model and the systems of annihilating random walks can then be constructed from the resulting percolation structure by
 assuming that arrows which are not active have no effect on any of the two processes whereas if the $i$-arrow \eqref{eq:arrow} is active \eqref{eq:active} then
\begin{itemize}
 \item at time $t$, the individual at vertex~$y$ looks at the individual at vertex~$x$ and imitates her opinion for the $i$th issue,
   therefore we set $\bar \eta_t (y, i) := \bar \eta_{t-} (x, i)$, \vspace*{4pt}
 \item the particle at $x + (1/2) \,B_{x, i} (n)$ at level~$i$ jumps to $x + (3/2) \,B_{x, i} (n)$.
\end{itemize}
 To establish our coexistence results, it is also useful to identify the vertices where the opinions of the profile of some given vertex at some given time originate from.
 This can be done looking at {\bf active paths}:
 we say that there is an active $i$-path from $(z, s)$ to $(x, t)$ whenever there are sequences of times and vertices
 $$ s_0 \ = \ s \ < \ s_1 \ < \ \cdots \ < \ s_{n + 1} \ = \ t \qquad \hbox{and} \qquad
    x_0 \ = \ z, \,x_1, \,\ldots, \,x_n \ = \ x $$
 such that the following two conditions hold:
\begin{enumerate}
 \item For $j = 1, 2, \ldots, n$, there is an active $i$-arrow $x_{j - 1} \to x_j$ at time $s_j$. \vspace*{4pt}
 \item For $j = 0, 1, \ldots, n$, there is no active $i$-arrow that points at $\{x_j \} \times (s_j, s_{j + 1})$.
\end{enumerate}
 We say that there is a generalized active path from $(z, s)$ to $(x, t)$ whenever
\begin{enumerate}
 \item[3.] For $j = 1, 2, \ldots, n$, there is an active arrow $x_{j - 1} \to x_j$ at time $s_j$.
\end{enumerate}
 Later, we will use the notations $\overset{i}{\leadsto}$ and $\leadsto$ to indicate the existence of an active $i$-path and
 a generalized active path.
 We also point out that for every space-time point $(x, t)$ there is a unique space-time point $(z, 0)$ such that both points are
 connected by an active $i$-path which, using a simple induction, implies that the corresponding individuals at the corresponding
 times agree on the $i$th issue, i.e., $z$ is the ancestor of $(x, t)$ for the $i$th issue.




\section{Proof of Theorem~\ref{th:flux}}
\label{sec:flux}

\noindent In this section, we prove that, when~$F = \theta + 1$, the opinion model clusters.
 To begin with, we note that, in this parameter region, the vectorial Deffuant model is closely related to the two-state Axelrod model~\cite{axelrod_1997}.
 Indeed, the one-dimensional construction in the previous section can be applied to the latter, which again results in a system of non-independent annihilating
 random walks.
 The only difference between the two models is that, in the system of random walks coupled with the Axelrod model, at each edge
 occupied by~$j$~particles, these particles jump at rate
\begin{equation}
\label{eq:rate-axelrod}
  \begin{array}{rclcl}
   r_{\ax} (j) & = & j^{-1} (1 - j / F) & \hbox{when} & j \neq 0. \end{array}
\end{equation}
 In particular, when $F = \theta + 1$, it follows from \eqref{eq:rate-deffuant} and \eqref{eq:rate-axelrod} that, for both models,
 an edge is a blockade if and only if it has exactly $F$ particles:
 $$ r (j) = 0 \qquad \hbox{if and only if} \qquad r_{\ax} (j) = 0 \qquad \hbox{if and only if} \qquad j = F. $$
 Now, clustering of the two-state Axelrod model has been proved in \cite{lanchier_schweinsberg_2012} and, while it heavily
 relies on the fact that an edge is a blockade if and only if it has $F$ particles, the proof is not sensitive to the exact rate
 at which active particles jump.
 In particular, Theorem~\ref{th:flux} directly follows from the arguments introduced in~\cite{lanchier_schweinsberg_2012} for the
 Axelrod model.
 In addition to the coupling with annihilating random walks, there are two key ingredients:
 each blockade breaks eventually with probability one and, as a consequence, the system of active and frozen particles goes extinct.
 Below, we only give the idea of the proof and refer to Sections 3 and 4 in~\cite{lanchier_schweinsberg_2012} for more details. \vspace*{-5pt} \\


\noindent {\bf Blockade destruction} --
 The first step is to prove destruction of the blockades:
 assuming that a designated edge $e_{\star}$ is a blockade at some time $t$, we have
\begin{equation}
\label{eq:extinction-1}
  T \ := \ \inf \,\{s > t : \zeta_s (e_{\star}) \neq F \} \ < \ \infty \quad \hbox{with probability one}.
\end{equation}
 The proof of~\eqref{eq:extinction-1} relies on two ingredients: parity preserving of the number of particles at each level and a symmetry
 argument introduced by Adelman \cite{adelman_1976} to show site recurrence of systems of annihilating random walks.
 To briefly explain parity preserving, assume that
\begin{itemize}
 \item edge $e^{\star}$ with $e^{\star} > e_{\star}$ also is a blockade at time $t$, i.e., $\zeta_t (e^{\star}) = F$, and \vspace*{4pt}
 \item between the two blockades $e_{\star}$ and $e^{\star}$, the number of particles at some level~$i$ and the number of particles at some other level~$j$ do not have the same parity, i.e.,
 \begin{equation}
 \label{eq:parity}
    \begin{array}{l} \sum_{e_{\star} \leq e \leq e^{\star}} \ \xi_t (e, i) \ \neq \ \sum_{e_{\star} \leq e \leq e^{\star}} \ \xi_t (e, j) \mod 2 \quad \hbox{for some} \quad i \neq j. \end{array}
 \end{equation}
\end{itemize}
 Now, let $\tau$ be the first time one of these two blockades breaks.
 Since particles at the same level annihilate by pairs, the parity of the number of particles between the two blockades is preserved at each level and up to time $\tau$.
 This, together with \eqref{eq:parity}, implies that, up to time $\tau$, there is at least one active particle between the two blockades so either this particle or another active particle
 outside the interval breaks one of the blockades after a finite time:
\begin{equation}
\label{eq:extinction-2}
  \tau \ := \ \inf \,\{s > t : \zeta_s (e_{\star}) \neq F \ \hbox{or} \ \zeta_s (e^{\star}) \neq F \} \ < \ \infty \quad \hbox{with probability one}.
\end{equation}
 The property in \eqref{eq:extinction-1} can be deduced from its analog~\eqref{eq:extinction-2} for two blockades also using some symmetry arguments through the following construction.
 First, we introduce
 $$ \begin{array}{rcl}
     B_0 & := & \{e_{\star}, e_{\star} + 1, \ldots, e_*, e_* + 1, e_* + 2 \} \quad \hbox{where} \vspace*{4pt} \\
     e_* & := & \min \,\{e > e_{\star} : e \ \hbox{and} \ e + 1 \ \hbox{and} \ e + 2 \ \hbox{have not been updated by time} \ t \}. \end{array} $$
 Next, we partition the half-line into intervals with the same length as $B_0$ by setting
 $$ B_n \ := \ B_0 + (e_* + 3 - e_{\star}) \,n \quad \hbox{for all} \quad n \geq 1 $$
 and let $N$ be the smallest $n$ such that
\begin{itemize}
 \item the configurations of particles in $B_0$ and $B_n$ at time $t$ can be obtained from one another by translation or reflection and \vspace*{4pt}
 \item none of the edges in $B_n$ has been updated by time $t$.
\end{itemize}
 Then, letting $e^{\star}$ be the rightmost edge in $B_N$,
\begin{enumerate}
 \item[(a)] the probability that~\eqref{eq:parity} holds, in which case \eqref{eq:extinction-2} holds as well, is $\geq 1/2$, \vspace*{4pt}
 \item[(b)] the probability that the configurations of particles in $B_0$ and $B_N$ at time $t$ can be obtained from one another by reflection is $\geq 1/2$, and \vspace*{4pt}
 \item[(c)] by symmetry, the conditional probability given \eqref{eq:parity} and reflection that the blockade~$e_{\star}$ breaks before the blockade~$e^{\star}$ is equal to~1/2.
\end{enumerate}
 From (a)--(c), we deduce that $P \,(T = \tau) \geq (1/2)^3 = 1/8$.
\begin{itemize}
 \item In case the events in (a) or (b) do not occur, we repeat the same construction starting from the same time~$t$ but replacing the interval~$B_0$ with $B'_0 :=$ the smallest interval that
  contains~$B_N$ and whose rightmost three edges have not been updated by time~$t$. \vspace*{4pt}
 \item In case the events in (a) and (b) occur but not the one in (c), we repeat the same construction but starting from the time the blockade~$e^{\star}$ breaks and replacing~$B_0$ with $B'_0$.
\end{itemize}
 After a geometric number of steps with success probability 1/8, all three events in (a)--(c) occur, from which it follows that time $T$ is almost surely finite. \vspace*{-5pt} \\


\noindent {\bf Extinction of the particles} --
 To complete the proof of the theorem, it suffices to show extinction of the system of random walks, since this property is equivalent to clustering of the original
 opinion model.
 The proof deals with active particles and frozen particles separately.
 First, we assume by contradiction that the expected number of active particles at edge $e$, which does not depend on the choice of $e$ due to translation invariance,
 does not converge:
\begin{equation}
\label{eq:contradict-limit}
  \begin{array}{l} \limsup_{t \to \infty} \ E \,(\zeta_t (e) \,\ind \{\zeta_t (e) \neq F \}) \ \neq \ \liminf_{t \to \infty} \ E \,(\zeta_t (e) \,\ind \{\zeta_t (e) \neq F \}). \end{array}
\end{equation}
 Since the expected number of active particles per edge can only decrease due to annihilating events or active particles becoming frozen, the fact that the expected number of
 active particles per edge goes infinitely often from the $\limsup$ to the $\liminf$ in \eqref{eq:contradict-limit} implies that
\begin{equation}
\label{eq:annihilating-freezing}
  \begin{array}{l} \lim_{t \to \infty} \ E \,(\zeta_t (e)) \ = \ - \infty \quad \hbox{or} \quad \lim_{t \to \infty} \ E \,(\zeta_t (e) \,\ind \{\zeta_t (e) = F \}) = + \infty, \end{array}
\end{equation}
 which is not possible.
 In particular, \eqref{eq:contradict-limit} is not true.
 Now, we assume by contradiction that the expected number of active particles per edge converges to a positive limit:
\begin{equation}
\label{eq:contradict-active}
  \begin{array}{l} \lim_{t \to \infty} \ E \,(\zeta_t (e) \,\ind \{\zeta_t (e) \neq F \}) \ = \ \ep \ > \ 0. \end{array}
\end{equation}
 Since one-dimensional symmetric random walks are recurrent, each active particle either gets annihilated or becomes frozen eventually with probability one
 therefore \eqref{eq:contradict-active} implies that, at all times, the expected number of annihilating events per edge per unit of time or the expected number of freezing
 events per edge per unit of time is larger than some positive constant, which again leads to the impossible statement~\eqref{eq:annihilating-freezing}.
 It follows that~\eqref{eq:contradict-active} is not true.
 To deal with the frozen particles, we first observe that, since the expected number of particles per edge can only decrease, it has a limit as time goes to infinity.
 This, together with the fact that the expected number of active particles per edge has a limit, implies that the expected number of frozen particles per edge has a limit
 as well, and we assume by contradiction that this limit is positive:
\begin{equation}
\label{eq:contradict-frozen}
  \begin{array}{l} \lim_{t \to \infty} \ E \,(\zeta_t (e) \,\ind \{\zeta_t (e) = F \}) \ = \ \ep \ > \ 0. \end{array}
\end{equation}
 Since each blockade breaks eventually with probability one according to~\eqref{eq:extinction-1} and since each blockade destruction results in the annihilation of two
 particles, \eqref{eq:contradict-frozen} implies that the expected number of annihilating events per edge per unit of time is larger than a positive constant,
 thus leading to the left-hand side of~\eqref{eq:annihilating-freezing}, again a contradiction.
 Therefore, the frozen particles go extinct.


\section{Proof of Theorem~\ref{th:uniform-fix}}\label{sec:uniform-fix}

\noindent This section is devoted to the study of the coexistence regime for the system starting from the uniform product
 measure~\eqref{eq:uniform}.
 The first ingredient is the next lemma, which gives a sufficient condition for fixation and is similar to
 Lemma~2 in~\cite{bramson_griffeath_1989} where it is applied to cyclic particle systems.
\begin{lemma} --
\label{lem:fixation}
 For all $(z, i) \in \Z \times \{1, 2, \ldots, F \}$, let
 $$ T (z, i) \ := \ \inf \,\{t : (z, 0) \overset{i}{\leadsto} (0, t) \}. $$
 Then, the system fixates whenever
\begin{equation}
\label{eq:fixation-1}
 \begin{array}{l}
   \lim_{N \to \infty} \,P \,(T (z, i) < \infty \ \hbox{for some} \ z < - N \ \hbox{and some} \ i = 1, 2, \ldots, F) \ = \ 0.
 \end{array}
\end{equation}
\end{lemma}
\begin{proof}
 This follows exactly the proof of Lemma~4 in reference~\cite{lanchier_scarlatos_2013}.
\end{proof} \\ \\
 To motivate the structure of the proof, we first use Lemma~\ref{lem:fixation} to explain in detail the connection between
 the initial configuration of the system, i.e., the initial number of active particles and the initial number of frozen particles,
 and the key event
 $$ H_N \ := \ \{T (z, i) < \infty \ \hbox{for some} \ z < - N \ \hbox{and some} \ i = 1, 2, \ldots, F \} $$
 that appears in~\eqref{eq:fixation-1}.
 Following the same approach as~\cite{lanchier_scarlatos_2013}, we let $\tau$ be the first time an active $i$-path that
 originates from the interval $(- \infty, - N)$ hits the origin, and observe that
 $$ \tau \ = \ \inf \,\{T (z, i) : z \in (- \infty, - N) \ \hbox{and} \ i = 1, 2, \ldots, F \} $$
 from which it follows that $H_N$ can be written as
\begin{equation}
\label{eq:key-event}
 \begin{array}{rcl}
   H_N & = & \{T (z, i) < \infty \ \hbox{for some} \ (z, i) \in (- \infty, - N) \times \{1, 2, \ldots, F \} \} \vspace*{4pt} \\
       & = & \{\tau < \infty \}. \end{array}
\end{equation}
 Letting $z^{\star} < - N$ be the initial position of this active path and
\begin{equation}
\label{eq:paths}
  \begin{array}{rcl}
    z_- & := & \min \,\{z \in \Z : (z, 0) \leadsto (0, \tau) \} \ \leq \ z^{\star} \ < \ - N \vspace*{2pt} \\
    z_+ & := & \max \,\{z \in \Z : (z, 0) \leadsto (0, \sigma) \ \hbox{for some} \ \sigma < \tau \} \ \geq \ 0 \end{array}
\end{equation}
 and defining $I = (z_-, z_+)$, we have the following two properties:
\begin{itemize}
\item All the blockades initially in~$I$ must have been destroyed, i.e., turned into piles of $\theta$~active
      particles due to annihilating events, by time $\tau$. \vspace*{4pt}
\item By definition of $z_-$ and $z_+$, the active particles initially outside $I$ cannot jump inside the space-time
      region delimited by the two generalized active paths defined in~\eqref{eq:paths}.
\end{itemize}
 This together with~\eqref{eq:key-event} implies that, on the event~$H_N$, all the blockades initially in~$I$ must have
 been destroyed before time~$\tau$ by either active particles initially in~$I$ or active particles that result from the destruction
 of these blockades.
 In particular, introducing the following random variables, that we shall call {\bf contributions}, which are measurable with respect to the initial
 configuration and the graphical representation of the process and defined for each edge $e$ as
\begin{equation}
\label{eq:contribution-active}
 \begin{array}{rcl}
     \cont (e) & := & \hbox{number of active particles that either annihilate or become frozen} \\ &&
                      \hbox{as the result of a jump onto $e$ before the first jump of an active} \\ &&
                      \hbox{particle initially at $e$ minus the number of particles initially at $e$} \end{array}
\end{equation}
 when $e$ is initially a live edge, and
\begin{equation}
\label{eq:contribution-frozen}
 \begin{array}{rcl}
     \cont (e) & := & \hbox{number of active particles that either annihilate or become frozen} \\ &&
                      \hbox{as the result of a jump onto $e$ before $e$ becomes a live edge minus} \\ &&
                      \hbox{the number of particles initially at $e$ that ever become active} \end{array}
\end{equation}
 when $e$ is initially a blockade,
 we obtain the following inclusions
\begin{equation}
\label{eq:inclusion}
 \begin{array}{rcl}
   H_N & \subset & \{\sum_{e \in I} \,\cont (e) \leq 0 \} \vspace*{4pt} \\
       & \subset & \{\sum_{e \in (l, r)} \cont (e) \leq 0 \ \hbox{for some $l < - N$ and some $r \geq 0$} \}. \end{array}
\end{equation}
 We now briefly describe the structure of our proof to deduce fixation and coexistence.
 The first step is to find an explicit random function, that we shall call {\bf weight}, defined on the edge set and which is stochastically
 smaller than the contribution random variable.
 Then, proving large deviation estimates for the total weight of a large interval and using Lemma~\ref{lem:fixation} and~\eqref{eq:inclusion},
 we will deduce that fixation occurs whenever the expected value of the weight at a single edge is strictly positive.
 To complete the proof, we will exploit the symmetry of the probability mass function of the binomial random variable to study the sign
 of the expected weight from which fixation will follow for the parameter region described in the statement of the theorem with the
 exception of the three-feature system with threshold one.
 To study this last case, we will improve our stochastic bound for the contribution by also accounting for pairs of active particles
 forming blockades.
\begin{lemma} --
\label{lem:blockade-weight}
 The contribution $\cont (e)$ is stochastically larger than
 $$ \begin{array}{rclclcl}
     \phi (e) & := & - j & \hbox{when} & \zeta_0 (e) = j \leq \theta \vspace*{4pt} \\
              & := &   j + 2 \,(X_j - \theta) & \hbox{when} & \zeta_0 (e) = j > \theta & \hbox{where} & X_j := \bernoulli (1 - j / F). \end{array} $$
\end{lemma}
\begin{proof}
 Let $j$ be the initial number of particles at $e$ and assume first that $j \leq \theta$.
 In this case, the edge is initially a live edge therefore~\eqref{eq:contribution-active} implies that
 $$ \begin{array}{rcl}
     \cont (e) & \geq & \hbox{minus the number of particles initially at $e$} \ = \ - j \end{array} $$
 almost surely, which proves the first part of the lemma.
 Now, assume that $j > \theta$, implying that the edge is initially a blockade.
 Then, observe that $j - \theta$~active particles must annihilate with some of the frozen particles of the blockade to break the blockade, and that
 this results in a total of exactly~$\theta$~frozen particles initially at edge~$e$ becoming active.
 This together with~\eqref{eq:contribution-frozen} gives the following lower bound for the contribution:
 $$ \begin{array}{rcll}
     \cont (e) & \geq & (j - \theta) - \theta \ = \ j - 2 \theta & \hbox{almost surely}. \end{array} $$
 The last step to improve the bound as indicated in the statement of the lemma is to also estimate the number of active particles that become frozen
 as the result of a jump onto $e$ before the blockade becomes a live edge.
 More precisely, we look at the probability that the first jump of an active particle onto the blockade results in an annihilating event, which can
 be computed explicitly using the following symmetry argument:
 since both the initial distribution and the dynamics of the model are invariant by permutation of the levels, and since the configuration of
 particles outside~$e$ is independent of the distribution of particles at~$e$ by the time of the first jump of an active particle onto~$e$, this
 first jump occurs with equal probability at each level.
 In particular, the first jump of an active particle onto the blockade~$e$ results in either
\begin{itemize}
\item an annihilating event with probability $j/F$, in which case the number of active particles required to break the blockade is the same as before or \vspace*{4pt}
\item a blockade increase with probability $1 - j/F$, in which case one active particle becomes frozen and one additional active particle is required
  to eventually break the blockade.
\end{itemize}
 Since two additional active particles are eliminated in the event of a blockade increase, we deduce that the contribution
 of $e$ is stochastically larger than
 $$ (1 - X_j)(j - 2 \theta) + X_j \,(j - 2 \theta + 2) \ = \ j + 2 \,(X_j - \theta) $$
 where $X_j = \bernoulli (1 - j / F)$.
 This completes the proof.
\end{proof} \\ \\
 In the next lemma, we prove large deviation estimates for the weight in a large interval, from which we deduce, in the subsequent lemma, that
 the system fixates whenever the expected value of the weight function at a given edge is strictly positive.
\begin{lemma} --
\label{lem:large-deviations}
 There exist $C_1 < \infty$ and $c_1 > 0$ such that, for all $\ep > 0$,
 $$ \begin{array}{l} P \,(\sum_{e \in (- N, 0)} \phi (e) \leq N (E \phi (e) - \ep)) \ \leq \ C_1 \exp (- c_1 N \ep^2). \end{array} $$
\end{lemma}
\begin{proof}
 The idea is to prove that the number of $j$-edges in a given interval is a binomial random variable and then apply the standard large deviation estimates: for all $\ep > 0$,
\begin{equation}
\label{eq:uniform-fix-1}
 \begin{array}{rcl}
   P \,(Z \leq N (p - \ep)) & \leq & \exp (- (1/2) \,N p^{-1} \ep^2) \ \leq \ \exp (- (1/2) \,N \ep^2) \vspace*{4pt} \\
   P \,(Z \geq N (p + \ep)) & \leq & \exp (- (1/2) \,N (1 - p)^{-1} \ep^2) \ \leq \ \exp (- (1/2) \,N \ep^2)
 \end{array}
\end{equation}
 where $Z = \binomial (N, p)$.
 First, we note that, starting from the uniform product measure, the opinions at two adjacent vertices at a given level are initially either equal or different
 with probability one half, independently of the rest of the initial configuration.
 This implies that the initial number of particles at any given edge is a binomial random variable:
\begin{equation}
\label{eq:uniform-fix-2}
  p_j \ := \ P \,(\zeta_0 (e) = j) \ = \ {F \choose j} (1/2)^F \quad \hbox{for each edge} \ e.
\end{equation}
 Recalling the expression of the bound $\phi (e)$, we thus obtain
\begin{equation}
\label{eq:uniform-fix-3}
 \begin{array}{rcl}
   E \phi (e) & = & \sum_{j \leq \theta} \,(- j) \,p_j + \sum_{j > \theta} \,(j + 2 \,(1 - j / F - \theta)) \,p_j \vspace*{4pt} \\
              & = & \sum_{j \leq \theta} \,(- j) \,p_j + \sum_{j > \theta} \,(j - 2 \theta) \,p_j + \sum_{j > \theta} \,2 \,(1 - j / F) \,p_j. \end{array}
\end{equation}
 To also have an explicit expression for the weight of $(- N, 0)$, we let
 $$ \Omega_j \ := \ \{e \in (- N, 0) : \zeta_0 (e) = j \} \quad \hbox{and} \quad e_N (j) \ := \ \card \,\Omega_j \quad \hbox{for} \quad j = 0, 1, \ldots, F, $$
 denote the set of and the number of $j$-edges in $(- N, 0)$, respectively.
 Using again that initially the different pairs edge-level are independently empty or occupied by a particle with equal probability, a simple extension of
 the symmetry argument from Lemma~\ref{lem:blockade-weight} implies that the Bernoulli random variables that determine the outcome of the first jump onto
 the blockades are independent.
 It follows that the random weight of the interval $(- N, 0)$ can be expressed as
\begin{equation}
\label{eq:uniform-fix-4}
 \begin{array}{rcl}
  \sum_{e \in (- N, 0)} \phi (e) & = & \sum_{j \leq \theta} \,(-j) \,e_N (j) + \sum_{j > \theta} \,(j + 2 \,(X_{e, j} - \theta)) \,e_N (j) \vspace*{4pt} \\
                            & = & \sum_{j \leq \theta} \,(-j) \,e_N (j) + \sum_{j > \theta} \,(j - 2 \theta) \,e_N (j) + \sum_{j > \theta} \,\sum_{e \in \Omega_j} 2 X_{e, j} \end{array}
\end{equation}
 where the random variables $X_{e, j}$ are independent Bernoulli random variables with the same success probability $1 - j / F$.
 Combining the expressions~\eqref{eq:uniform-fix-3} and~\eqref{eq:uniform-fix-4}, we deduce that
\begin{equation}
\label{eq:uniform-fix-5}
 \begin{array}{l}
   P \,(\sum_{e \in (- N, 0)} \phi (e) \leq N (E \phi (e) - \ep)) \vspace*{4pt} \\ \hspace*{20pt} \leq \
        \sum_{j \leq \theta} \,P \,((-j)(e_N (j) - N p_j) \leq - N \ep / 2F) \vspace*{4pt} \\ \hspace*{20pt} + \
        \sum_{j > \theta} \,P \,((j - 2 \theta)(e_N (j) - N p_j) \leq - N \ep / 2F) \vspace*{4pt} \\ \hspace*{20pt} + \
        \sum_{j > \theta} \,P \,(\sum_{j \in \Omega_j} 2 X_{e, j} - 2 \,(1 - j / F) N p_j \geq N \ep / 2F). \end{array}
\end{equation}
 To bound the first two terms, we first use~\eqref{eq:uniform-fix-2} and independence to deduce
\begin{equation}
\label{eq:uniform-fix-6}
  e_N (j) \ = \ \binomial (N, p_j) \quad \hbox{for all} \quad j = 0, 1, \ldots, F.
\end{equation}
 Then, using~\eqref{eq:uniform-fix-1} and~\eqref{eq:uniform-fix-6}, we get
\begin{equation}
\label{eq:uniform-fix-7}
 \begin{array}{l}
   P \,((-j)(e_N (j) - N p_j) \leq - N \ep / 2F) \vspace*{4pt} \\ \hspace*{40pt} \leq \
   P \,(e_N (j) - N p_j \geq N \ep / 2F^2) \ \leq \ \exp (- N \ep^2 / 8 F^4) \end{array}
\end{equation}
 for $j = 1, 2, \ldots, \theta$, and
\begin{equation}
\label{eq:uniform-fix-8}
 \begin{array}{l}
   P \,((j - 2 \theta)(e_N (j) - N p_j) \leq - N \ep / 2F) \vspace*{4pt} \\ \hspace*{40pt} \leq \
   P \,(e_N (j) - N p_j \notin (- N \ep / 2F^2, N \ep / 2F^2)) \ \leq \ 2 \,\exp (- N \ep^2 / 8 F^4) \end{array}
\end{equation}
 for $j = \theta + 1, \ldots, F$.
 Finally, using again the second inequality in~\eqref{eq:uniform-fix-1} together with the fact that the random variables~$X_{e, j}$ are independent, we get
 $$ \begin{array}{l}
     P \,(\sum_{e \in \Omega_j} 2 X_e - 2 \,(1 - j / F) N p_j \geq N \ep / 2F \,| \,e_N (j) < N \,(p_j + \ep / 16F)) \vspace*{4pt} \\ \hspace*{20pt} \leq \
     P \,(\sum_{e \in \Omega_j} X_e - (1 - j / F) N p_j \geq N \ep / 4F \,| \,\card \Omega_j = N \,(p_j + \ep / 16F)) \vspace*{4pt} \\ \hspace*{20pt} \leq \
     P \,(\binomial (N (p_j + \ep / 16F), 1 - j/F) \geq N (\ep / 4F + (1 - j / F) \,p_j)) \vspace*{4pt} \\ \hspace*{20pt} \leq \
     P \,(\binomial (N (p_j + \ep / 16F), 1 - j/F) \geq N (p_j + \ep / 16F)(1 - j/F + \ep / 16F)) \vspace*{4pt} \\ \hspace*{20pt} \leq \
       \exp (- (1/2) \,N \,(p_j + \ep / 16F)(\ep / 16F)^2) \end{array} $$
 from which we deduce that
\begin{equation}
\label{eq:uniform-fix-9}
 \begin{array}{l}
  P \,(\sum_{e \in \Omega_j} 2 X_e - 2 \,(1 - j / F) N p_j \geq N \ep / 2F) \vspace*{4pt} \\ \hspace*{20pt} \leq \
  P \,(\sum_{e \in \Omega_j} 2 X_e - 2 \,(1 - j / F) N p_j \geq N \ep / 2F \,| \,e_N (j) < N \,(p_j + \ep / 16F)) \vspace*{4pt} \\ \hspace*{40pt} + \
  P \,(e_N (j) \geq N \,(p_j + \ep / 16F)) \vspace*{4pt} \\ \hspace*{20pt} \leq \
    \exp (- (1/2) \,N \,(p_j + \ep / 16F)(\ep / 16F)^2) \ + \ \exp (- (1/2) \,N \,(\ep / 16F)^2).
 \end{array}
\end{equation}
 The lemma then follows from~\eqref{eq:uniform-fix-5} and~\eqref{eq:uniform-fix-7}--\eqref{eq:uniform-fix-9}.
\end{proof}
\begin{lemma} --
\label{lem:uniform-fix}
 The system fixates whenever $E \phi (e) > 0$.
\end{lemma}
\begin{proof}
 Let $\ep := E \phi (e) > 0$.
 Then, according to Lemma~\ref{lem:large-deviations},
 $$ \begin{array}{rcl}
     P \,(\sum_{e \in (- N, 0)} \phi (e) \leq 0) & = & P \,(\sum_{e \in (- N, 0)} \phi (e) \leq N (E \phi (e) - \ep)) \vspace*{4pt} \\
                                                 & \leq & C_1 \exp (- c_1 N \ep^2). \end{array} $$
 This, together with~\eqref{eq:inclusion} and Lemma~\ref{lem:blockade-weight}, implies that
 $$ \begin{array}{l}
      \lim_{N \to \infty} P \,(H_N) \
      \leq \ \lim_{N \to \infty} P \,(\sum_{e \in I} \,\cont (e) \leq 0) \vspace*{4pt} \\ \hspace*{20pt}
      \leq \ \lim_{N \to \infty} P \,(\sum_{e \in (l, r)} \cont (e) \leq 0 \ \hbox{for some $l < - N$ and some $r \geq 0$}) \vspace*{4pt} \\ \hspace*{20pt}
      \leq \ \lim_{N \to \infty} P \,(\sum_{e \in (l, r)} \phi (e) \leq 0 \ \hbox{for some $l < - N$ and some $r \geq 0$}) \vspace*{4pt} \\ \hspace*{20pt}
      \leq \ \lim_{N \to \infty} \sum_{l < - N} \,\sum_{r > 0} \,P \,(\sum_{e \in (l, r)} \phi (e) \leq 0) \vspace*{4pt} \\ \hspace*{20pt}
      \leq \ \lim_{N \to \infty} \sum_{l < - N} \,\sum_{r > 0} \,C_1 \exp (- c_1 \,(r - l)\ep^2) \ = \ 0. \end{array} $$
 In particular, Lemma \ref{lem:fixation} implies that the system fixates.
\end{proof} \\ \\
 Having Lemma~\ref{lem:uniform-fix} in hands, the last step is to exhibit the set of parameters for which the expected value of the weight
 function at a single edge is positive.
 This can be done by making the expression of the expected value more explicit but this leads to messy calculations.
 As previously mentioned, we use instead that the probability mass function of the binomial random variable is symmetric when the success
 probability is equal to one half.
 Our approach is illustrated in Figure~\ref{fig:weight} where the left-hand side represents the expected value of the weight as a function
 of the initial number of particles at the edge and where the right-hand side is obtained by folding this picture along the vertical dashed
 line and by adding the expected values.
 The figure suggests that, through this simple transformation, the expected value of the weight can be expressed as a sum of positive values
 when $F \geq 4 \theta$, which is done rigorously in the next lemma.
\begin{lemma} --
\label{lem:inner}
 Assume~\eqref{eq:uniform} and $F \geq 4 \theta$. Then, $E \phi (e) > 0$.
\end{lemma}
\begin{proof}
 To begin with, we introduce
 $$ \begin{array}{rcl}
     K_- & := & \hbox{the largest integer smaller than or equal to} \ (1/2)(F - 1) \vspace*{2pt} \\
     K_+ & := & \hbox{the smallest integer larger than or equal to} \ (1/2)(F + 1) \end{array} $$
 and observe that $K_- + K_+ = F$ and
\begin{equation}
\label{eq:inner-1}
  \begin{array}{rcll}
    K_+ - K_- & = & 1 & \hbox{when $F$ is odd} \vspace*{2pt} \\
              & = & 2 & \hbox{when $F$ is even}. \end{array}
\end{equation}
 Letting $q_j (\theta, F) := 2 \,(1 - \theta) + (1 - 2 / F) \,j$, we also have
 $$ \begin{array}{rcl}
                        q_{F - j} (\theta, F) & = & 2 \,(1 - \theta) + (1 - 2 / F)(F - j) \ = \ F - 2 \theta - (1 - 2 / F) \,j \vspace*{4pt} \\
      q_j (\theta, F) + q_{F - j} (\theta, F) & = & 2 \,(1 - \theta) + F - 2 \theta \ = \ F - 4 \theta + 2 \vspace*{4pt} \\
                        q_{F / 2} (\theta, F) & = & F/2 - 2 \theta + 1 \ \geq \ 1 \quad \hbox{for $F \geq 4 \theta$ and even}. \end{array} $$
 In particular, considering the intervals
\begin{equation}
\label{eq:inner-2}
  J_1 := [0, \theta], \quad J_2 := [\theta + 1, K_-], \quad J_3 := [K_+, F - (\theta + 1)], \quad J_4 := [F - \theta, F],
\end{equation}
 recalling~\eqref{eq:uniform-fix-3} and using the symmetry $p_j = p_{F - j}$, we obtain
\begin{figure}[t]
\centering
\scalebox{0.36}{\input{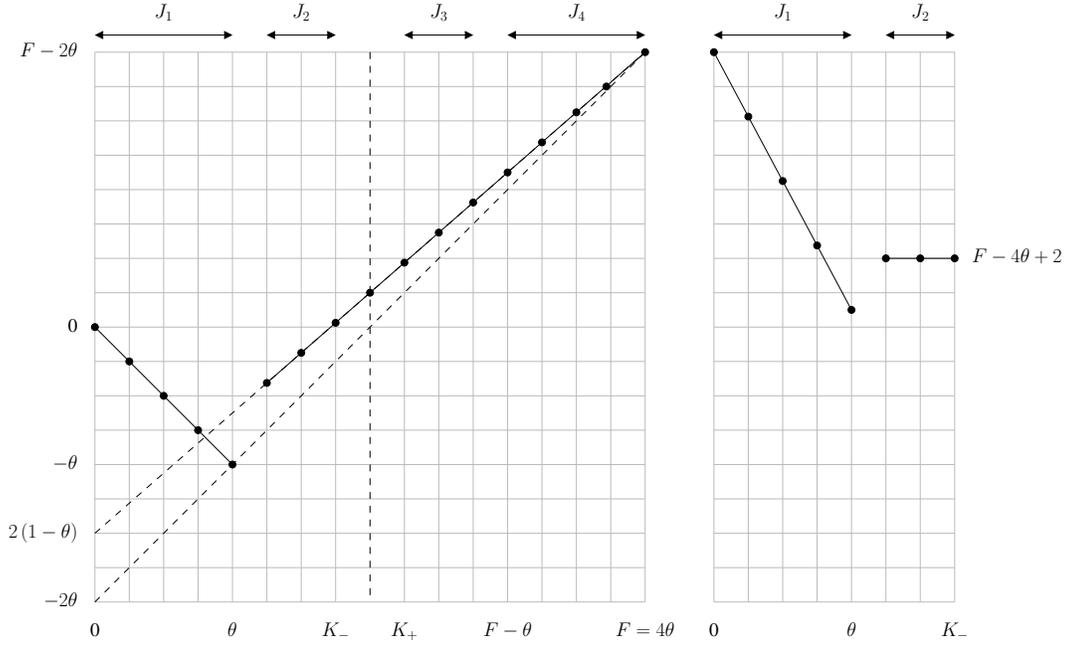}}
\caption{\upshape{Picture related to the proof of Lemma~\ref{lem:inner}.
 The dots on the left-hand side represent the value of the weight function.
 The right-hand side is obtained by using the symmetry of the probability mass function of the binomial random variable when the success
 probability is equal to one half and illustrates the next-to-last line in~\eqref{eq:inner-3}.}}
\label{fig:weight}
\end{figure}
\begin{equation}
\label{eq:inner-3}
 \begin{array}{rcl}
   E \phi (e) & = & \sum_{j \leq \theta} \,(- j) \,p_j + \sum_{j > \theta} \,(j + 2 \,(1 - j / F - \theta)) \,p_j \vspace*{4pt} \\
              & \geq & \sum_{j \in J_1} \,(- j) \,p_j + \sum_{k = 2, 3, 4} \,\sum_{j \in J_k} \,q_j (\theta, F) \,p_j \vspace*{4pt} \\
                 & = & \sum_{j \in J_1} \,((- j) + q_{F - j} (\theta, F)) \,p_j + \sum_{j \in J_2} \,(q_j (\theta, F) + q_{F - j} (\theta, F)) \,p_j \vspace*{4pt} \\
                 & = & \sum_{j \in J_1} \,(F - 2 \theta - 2 \,(1 - 1 / F) \,j) \,p_j + \sum_{j \in J_2} \,(F - 4 \theta + 2) \,p_j \vspace*{4pt} \\
              & \geq & \sum_{j \in J_1} \,(F - 2 \,(\theta + j)) \,p_j + \sum_{j \in J_2} \,(F - 4 \theta + 2) \,p_j \ > \ 0 \end{array}
\end{equation}
 for all $F \geq 4 \theta$ since in this case all the terms in the previous two sums are nonnegative with also some positive
 terms (see Figure~\ref{fig:weight} for a picture).
\end{proof} \\ \\
 The approach of the previous proof does not extend to the case $F = 4 \theta - 1$ because, for this set of parameters, the first sum in the last
 line of~\eqref{eq:inner-3} contains a negative term.
 To deal with this case avoiding messy calculations, we find a lower bound using the binomial random variable and then use standard large deviation
 estimates for this distribution.
\begin{lemma} --
\label{lem:boundary}
 Assume~\eqref{eq:uniform} and $F = 4 \theta - 1$ with $\theta \geq 2$. Then, $E \phi (e) > 0$.
\end{lemma}
\begin{proof}
 Let $F = 4 \theta - 1$ and observe that
 $$ \begin{array}{rclcl}
      F - 2 \,(\theta + j) & \geq & 4 \theta - 1 - 2 \,(\theta + \theta) \ = \ - 1 & \hbox{for all} & j \in J_1 \vspace*{2pt} \\
      F - 4 \theta + 2 & = & 4 \theta - 1 + 2 \ = \ 1 & \hbox{for all} & j \in J_2. \end{array} $$
 In particular, recalling~\eqref{eq:inner-3}, we obtain
 $$ \begin{array}{rcl}
      E \phi (e) & \geq & \sum_{j \in J_1} \,(F - 2 \,(\theta + j)) \,p_j + \sum_{j \in J_2} \,(F - 4 \theta + 2) \,p_j \vspace*{4pt} \\
                 & \geq & \sum_{j \in J_1} \,(- p_j) + \sum_{j \in J_2} \,p_j \ = \ P \,(Z \in J_2) - P \,(Z \in J_1) \vspace*{4pt} \\ \end{array} $$
 where $Z = \binomial (F, 1/2)$.
 Using that, according to~\eqref{eq:inner-1}, the intervals in~\eqref{eq:inner-2} form a partition of the range of the random
 variable~$Z$ when $F$ is odd, we obtain
 $$ \begin{array}{rcl}
      E \phi (e) & \geq & (1/2)(P \,(Z \in J_2) + P \,(Z \in J_3) - P \,(Z \in J_1) - P \,(Z \in J_4)) \vspace*{4pt} \\
                 & \geq & (1/2)(1 - 2 \,P \,(Z \in J_1) - 2 \,P \,(Z \in J_4)) \ = \ 1/2 - 2 \,P \,(Z \in J_1). \end{array} $$
 Then, using the standard large deviation estimate
 $$ P \,(Z \leq F (1/2 - \ep)) \ \leq \ \exp (- F \ep^2) \quad \hbox{for all} \quad \ep \in (0, 1/2) $$
 and taking $\ep = 13/54$, we deduce that, for all $\theta \geq 7$,
\begin{equation}
\label{eq:boundary-1}
 \begin{array}{rcl}
   E \phi (e) & \geq & 1/2 - 2 \,P \,(Z \leq \theta) \ \geq \ 1/2 - 2 \,P \,(Z \leq (4 \theta - 1)(1/2 - 13/54)) \vspace*{4pt} \\
              & \geq & 1/2 - 2 \,\exp (- (13/54)^2 (4 \theta - 1)) \ > \ 0. \end{array}
\end{equation}
 In addition, explicit calculations for $2 \leq \theta \leq 6$ show that
\begin{equation}
\label{eq:boundary-2}
   E \phi (e) \ \geq \ \bigg(3 \,{7 \choose 0} + \bigg(\frac{9}{7} \bigg) {7 \choose 1} - \bigg(\frac{3}{7} \bigg) {7 \choose 2} + {7 \choose 3} \bigg) \,(1/2)^7 \ = \ 19/64.
\end{equation}
 The lemma follows from combining~\eqref{eq:boundary-1}--\eqref{eq:boundary-2}.
\end{proof} \\ \\
 Putting together Lemmas~\ref{lem:uniform-fix}--\ref{lem:boundary}, we obtain the theorem except for the three-issue threshold one system
 in which case a direct calculation gives
 $$ E \phi (e) \ = \ - {3 \choose 1} (1/2)^3 + \bigg(\frac{1}{3} \bigg) {3 \choose 2} (1/2)^3 + {3 \choose 3} (1/2)^3 \ = \ 0. $$
 To also prove fixation when $\theta = 1$ and $F = 3$, the idea is to slightly improve the definition of our weight function to make
 its expected value strictly positive by also accounting for pairs of initially active particles that form a blockade before jumping onto
 a blockade.
\begin{lemma} --
\label{lem:threshold-one}
 The system with $\theta = 1$ and $F = 3$ fixates.
\end{lemma}
\begin{proof}
 We define the weight of a blockade as before by setting
\begin{equation}
\label{eq:threshold-one-1}
 \begin{array}{rclcl}
   \phi (e) & := & 2 + 2 \,(X_{e, 2} - 1) \ = \ 2 \,X_{e, 2}     & \hbox{when} & \zeta_0 (e) = 2  \vspace*{4pt} \\
            & := & 3 + 2 \,(X_{e, 3} - 1) \ = \ 2 \,X_{e, 3} + 1 & \hbox{when} & \zeta_0 (e) = 3 \end{array}
\end{equation}
 where the random variables $X_{e, j} = \bernoulli (1 - j / 3)$ are again independent.
 To improve our estimate for the weight of an edge initially occupied by an active particle, we take into account the possibility
 that, before it jumps, this active particle forms a blockade of size two with another active particle.
 Let $A_e$ be such an event for an active particle initially at $e$.
 To compute the probability of this event, we introduce the following two events:
 $$ \begin{array}{rcl}
      B_e^- & := & \hbox{there is initially an active particle at $e - 1$ which is not at the} \vspace*{0pt} \\ && \hbox{same level as the active particle initially at edge $e$} \vspace*{4pt} \\
      B_e^+ & := & \hbox{there is initially an active particle at $e + 1$ which is not at the} \vspace*{0pt} \\ && \hbox{same level as the active particle initially at edge $e$}. \end{array} $$
 Observe that, on the event $B_e^{\pm}$, the event $A_e$ occurs whenever the first jump of an active particle either directed to or starting
 from one of the two edges $e$ and $e \pm 1$ is a jump $e \pm 1 \to e$.
 Since all the active particles jump at the same rate, this gives
\begin{equation}
\label{eq:threshold-one-2}
 \begin{array}{rcl}
    P \,(A_e) & \geq & P \,(A_e \cap (B_e^- \setminus B_e^+)) + P \,(A_e \cap (B_e^+ \setminus B_e^-)) + P \,(A_e \cap (B_e^- \cap B_e^+)) \vspace*{4pt} \\
              & \geq & (1/6) \,P \,(B_e^- \setminus B_e^+) + (1/6) \,P \,(B_e^+ \setminus B_e^-) + (2/8) \,P \,(B_e^- \cap B_e^+). \end{array}
\end{equation}
 Independence and basic counting also imply that
\begin{equation}
\label{eq:threshold-one-3}
 \begin{array}{rcl}
    P \,(B_e^- \setminus B_e^+) & = & P \,(B_e^-) (1 - P \,(B_e^+)) \ = \ (2/8) \times (6/8) \ = \ 3/16  \vspace*{4pt} \\
    P \,(B_e^+ \setminus B_e^-) & = & P \,(B_e^+) (1 - P \,(B_e^-)) \ = \ (2/8) \times (6/8) \ = \ 3/16  \vspace*{4pt} \\
    P \,(B_e^- \cap B_e^+)      & = & P \,(B_e^-) \,P \,(B_e^+)     \ = \ (2/8) \times (2/8) \ = \ 1/16. \end{array}
\end{equation}
 Combining~\eqref{eq:threshold-one-2}--\eqref{eq:threshold-one-3}, we deduce that
 $$ P \,(A_e) \ \geq \ (1/6) \times (3/16) + (1/6) \times (3/16) + (2/8) \times (1/16) \ = \ 5/64. $$
 Note also that the events in~\eqref{eq:threshold-one-2} for different edges~$e$~and~$e'$ are independent whenever the two edges are at least
 distance four apart therefore our previous stochastic lower bound for the contribution of an active particle can be improved by setting
\begin{equation}
\label{eq:threshold-one-4}
 \begin{array}{rclcl}
   \phi (e) & := & - 1                       & \hbox{when} & \zeta_0 (e) = 1 \ \ \hbox{and} \ \ e + 1/2 \neq 0 \mod 4  \vspace*{4pt} \\
            & := & X_{e, 1} - (1 - X_{e, 1}) & \hbox{when} & \zeta_0 (e) = 1 \ \ \hbox{and} \ \ e + 1/2 = 0 \mod 4 \end{array}
\end{equation}
 where the random variables $X_{e, 1} = \bernoulli (5/64)$ are independent.
 Using the independence of these random variables, our proof of the large deviation estimates in Lemma~\ref{lem:large-deviations} easily
 extends to the weight function defined in~\eqref{eq:threshold-one-1}~and~\eqref{eq:threshold-one-4} and we get:
 for all $\ep > 0$,
 $$ \begin{array}{l} P \,(\sum_{e \in (- N, 0)} \phi (e) \leq N (m - \ep)) \ \leq \ C_1 \exp (- c_1 N) \end{array} $$
 for suitable $C_1 < \infty$ and $c_1 > 0$, where
 $$ \begin{array}{rrl}
      m & := & ((-3/4) + (1/4) \,E \,(2 X_{e, 1} - 1)) \,P \,(\zeta_0 (e) = 1) \vspace*{4pt} \\ && \hspace*{25pt} + \
                E \,(2 X_{e, 2}) \,P \,(\zeta_0 (e) = 2) \ + \ E \,(2 X_{e, 3} + 1) \,P \,(\zeta_0 (e) = 3) \vspace*{4pt} \\
        &  = & ((-3/4) + (1/4)(10 / 64 - 1))(3/8) \vspace*{4pt} \\ && \hspace*{25pt} + \
                2 \,(1 - 2/3)(3/8) + 1/8 \ = \ 15/1024 \ > \ 0. \end{array} $$
 In particular, fixation follows from the argument in the proof of Lemma~\ref{lem:uniform-fix}.
\end{proof} \\ \\
 The previous two lemmas imply fixation under the assumptions of Theorem~\ref{th:uniform-fix}.
 More precisely, our proof implies that each edge initially occupied by a blockade has a positive probability of never being updated
 which, in turn, implies that the two nearest neighbors on both sides of the blockade never update their opinion.
 Since each of the opinion profiles is equally likely to appear on both sides of the blockade, we deduce that the system fixates in
 a configuration where all the opinion profiles are present: the system coexists due to fixation.


\section{Proof of Theorem~\ref{th:biased-fix}}\label{sec:biased-fix}

\indent We now assume that the system starts from the product measure~\eqref{eq:biased}.
 Our approach to study fixation in this case is similar to the one in the previous section, the only additional difficulty being to
 extend the large deviation estimates in Lemma~\ref{lem:large-deviations} to non-uniform initial distributions where the number of
 particles at adjacent edges are no longer independent.
 In particular, the number of edges in a given interval and with a given initial number of particles is no longer a binomial random
 variable.
 In order to simplify the calculations, we define the weight function in the worst case scenario assuming that all the particles
 initially active never become frozen, i.e., we set all the random variables $X_j$ introduced in Lemma~\ref{lem:blockade-weight}
 equal to zero so that
\begin{equation}
\label{eq:biased-weight}
  \begin{array}{rclclc}
   \phi (e) & := & - j            & \hbox{when} & \zeta_0 (e) = j \leq \theta \vspace*{4pt} \\
            & := &   j - 2 \theta & \hbox{when} & \zeta_0 (e) = j > \theta. \end{array}
\end{equation}
 Denote the initial densities of opinion as
 $$ \rho (u) \ := \ P \,(\eta_0 (x) = u) \quad \hbox{for all} \quad u \in \Gamma. $$
 To extend Lemma~\ref{lem:large-deviations} to such product measures, we first study
 $$ e_N (u, v) \ := \ \card \,\{x \in [- N, 0] : \eta_0 (x) = u \ \hbox{and} \ \eta_0 (x + 1) = v \} \quad \hbox{for} \quad u, v \in \Gamma $$
 the number of edges connecting individuals with opinion~$u$~and~$v$, respectively.
 The next lemma gives large deviation estimates for the number of such edges which itself relies on large deviation estimates
 for the number of changeovers in a sequence of independent coin flips.
\begin{lemma} --
\label{lem:edge}
 There exist $C_2 < \infty$ and $c_2 > 0$ such that, for all $\ep > 0$ small,
 $$ P \,(e_N (u, v) - N \rho (u) \rho (v) \notin (- \ep N, \ep N)) \ \leq \ C_2 \exp (- c_2 N \ep^2) \quad \hbox{for all} \quad u \neq v. $$
\end{lemma}
\begin{proof}
 Let $X_0, X_1, \ldots, X_N$ be independent coin flips and let~$Z_N$ be the corresponding number of changeovers, i.e., the number of pairs of
 consecutive flips resulting in different outcomes:
 $$ Z_N \ := \ \card \,\{j = 0, 1, \ldots, N - 1 : X_j \neq X_{j + 1} \}. $$
 According to Lemma~7 in \cite{lanchier_scarlatos_2014}, there exist $C_3 < \infty$ and $c_3 > 0$ such that
\begin{equation}
\label{eq:edge-1}
 \begin{array}{l}
   P \,(Z_N - 2 N \,(1 - p) \,p \notin (- \ep N / 2, \ep N / 2)) \vspace*{4pt} \\ \hspace*{25pt} = \
   P \,(Z_N - E Z_N \notin (- \ep N / 2, \ep N / 2)) \ \leq \ C_3 \exp (- c_3 N \ep^2) \quad \hbox{for all} \quad \ep > 0 \end{array}
\end{equation}
 where $p$ is the probability that the coin comes up heads.
 Now, we observe that, for any $u$, the number of edges connecting an individual with initial opinion $u$ to an individual with a different
 opinion is equal in distribution to the number of changeovers when $p = \rho (u)$.
 In particular, the large deviation estimate in~\eqref{eq:edge-1} implies that, for all $u \in \Gamma$,
\begin{equation}
\label{eq:edge-2}
  \begin{array}{l} P \,(\sum_{v \neq u} \,e_N (u, v) - N \rho (u) \,(1 - \rho (u)) \notin (- \ep N / 2, \ep N / 2)) \ \leq \ C_3 \exp (- c_3 N \ep^2). \end{array}
\end{equation}
 Since each individual with initial opinion~$u$ preceding a changeover is independently followed by any of the remaining $2^F - 1$ opinions,
 we also have
\begin{equation}
\label{eq:edge-3}
  \begin{array}{l} e_N (u, v) = \binomial (K, \rho (v) (1 - \rho (u))^{-1}) \quad \hbox{on the event} \quad \sum_{w \neq u} \,e_N (u, w) = K. \end{array}
\end{equation}
 Letting $K_+ := N \rho (u) (1 - \rho (u)) + \ep N / 2$, observing that, for $\ep > 0$ small,
 $$ K_+ \,(\rho (v) (1 - \rho (u))^{-1} + (1/4) \,\rho (u)^{-1} (1 - \rho (u))^{-1} \,\ep) \ \leq \ N \,(\rho (u) \rho (v) + \ep) $$
 and combining \eqref{eq:edge-2}--\eqref{eq:edge-3} with the large deviation estimates~\eqref{eq:uniform-fix-1}, we get
\begin{equation}
\label{eq:edge-4}
  \begin{array}{l}
    P \,(e_N (u, v) - N \rho (u) \rho (v) \geq \ep N) \vspace*{4pt} \\ \hspace*{20pt} \leq \
    P \,(\sum_{w \neq u} \,e_N (u, w) - N \rho (u) (1 - \rho (u)) \geq \ep N / 2) \vspace*{4pt} \\ \hspace*{20pt} + \
    P \,(e_N (u, v) - N \rho (u) \rho (v) \geq \ep N \ | \ \sum_{w \neq u} \,e_N (u, w) - N \rho (u) (1 - \rho (u)) < \ep N / 2) \vspace*{4pt} \\ \hspace*{20pt} \leq \
    C_3 \exp (- c_3 N \ep^2) + P \,(\binomial (K_+, \rho (v) (1 - \rho (u))^{-1}) \geq N \,(\rho (u) \rho (v) + \ep)) \vspace*{4pt} \\ \hspace*{20pt} \leq \
    C_3 \exp (- c_3 N \ep^2) + \exp (- (1/32) \,\rho (u)^{-2} (1 - \rho (u))^{-2} \,K_+ \,\ep^2). \end{array}
\end{equation}
 Similarly, letting $K_- := N \rho (u) (1 - \rho (u)) - \ep N / 2$, we have
 $$ K_- \,(\rho (v) (1 - \rho (u))^{-1} - (1/4) \,\rho (u)^{-1} (1 - \rho (u))^{-1} \,\ep) \ \geq \ N \,(\rho (u) \rho (v) - \ep) $$
 and the same reasoning as in~\eqref{eq:edge-4} gives
\begin{equation}
\label{eq:edge-5}
  \begin{array}{l}
    P \,(e_N (u, v) - N \rho (u) \rho (v) \leq - \ep N) \vspace*{4pt} \\ \hspace*{20pt} \leq \
    C_3 \exp (- c_3 N \ep^2) + P \,(\binomial (K_-, \rho (v) (1 - \rho (u))^{-1}) \leq N \,(\rho (u) \rho (v) - \ep)) \vspace*{4pt} \\ \hspace*{20pt} \leq \
    C_3 \exp (- c_3 N \ep^2) + \exp (- (1/32) \,\rho (u)^{-2} (1 - \rho (u))^{-2} \,K_- \,\ep^2). \end{array}
\end{equation}
 The lemma follows from combining \eqref{eq:edge-4}--\eqref{eq:edge-5}.
\end{proof}
\begin{lemma} --
\label{lem:biased-fix}
 The system fixates whenever $E \phi (e) > 0$.
\end{lemma}
\begin{proof}
 For all $u, v \in \Gamma$, we set
 $$ \begin{array}{rclcl}
      h (u, v) & := & - H (u, v)          & \hbox{when} & H (u, v) = \card \,\{i : u_i \neq v_i \} \leq \theta  \vspace*{4pt} \\
               & := & H (u, v) - 2 \theta & \hbox{when} & H (u, v) = \card \,\{i : u_i \neq v_i \} > \theta \end{array} $$
 and observe that
 $$ \begin{array}{l}
     \sum_{e \in (- N, 0)} \,(\phi (e) - E \phi (e)) \ = \
     \sum_{e \in (- N, 0)} \phi (e) - N E \phi (e) \vspace*{4pt} \\ \hspace*{25pt} = \
     \sum_{u \neq v} \,h (u, v) \,e_N (u, v) - N \,\sum_{u \neq v} \,h (u, v) \,P \,(\eta_0 (x) = u \ \hbox{and} \ \eta_0 (x + 1) = v) \vspace*{4pt} \\ \hspace*{25pt} = \
     \sum_{u \neq v} \,h (u, v) \,(e_N (u, v) - N \rho (u) \rho (v)). \end{array} $$
 Then, letting $m := \max_{u, v} |h (u, v)| = \max (\theta, F - 2 \theta)$ and applying Lemma~\ref{lem:edge}, we get
\begin{equation}
\label{eq:biased-fix-1}
 \begin{array}{l}
   P \,(\sum_{e \in (- N, 0)} \,(\phi (e) - E \phi (e)) \notin (- \ep N, \ep N)) \vspace*{4pt} \\ \hspace*{20pt} = \
   P \,(\sum_{u \neq v} \,h (u, v) \,(e_N (u, v) - N \rho (u) \rho (v)) \notin (- \ep N, \ep N)) \vspace*{4pt} \\ \hspace*{20pt} \leq \
   P \,(e_N (u, v) - N \rho (u) \rho (v) \notin (- \ep N / m F^2, \ep N / m F^2) \ \hbox{for some} \ u \neq v) \vspace*{4pt} \\ \hspace*{20pt} \leq \
   C_2 \,F^2 \,\exp (- c_2 N \ep^2 / m^2 F^4) \end{array}
\end{equation}
 for all $\ep > 0$ small.
 Finally, we fix $\ep \in (0, E \phi (e))$ small enough so that \eqref{eq:biased-fix-1} holds and follow the same reasoning as
 in Lemma~\ref{lem:uniform-fix} to deduce that
 $$ \begin{array}{rcl}
      \lim_{N \to \infty} P \,(H_N) & \leq &
      \lim_{N \to \infty} P \,(\sum_{e \in (l, r)} \phi (e) \leq 0 \ \hbox{for some $l < - N$ and some $r \geq 0$}) \vspace*{4pt} \\ & \leq &
      \lim_{N \to \infty} \sum_{l < - N} \,\sum_{r > 0} \,P \,(\sum_{e \in (l, r)} \phi (e) \leq 0) \vspace*{4pt} \\ & \leq &
      \lim_{N \to \infty} \sum_{l < - N} \,\sum_{r > 0} \,P \,(\sum_{e \in (l, r)} (\phi (e) - E \phi (e)) \leq - \ep (r - l)) \vspace*{4pt} \\ & \leq &
      \lim_{N \to \infty} \sum_{l < - N} \,\sum_{r > 0} \,C_2 \,F^2 \,\exp (- c_2 (r - l) \ep^2 / m^2 F^4) \ = \ 0. \end{array} $$
 As in Lemma~\ref{lem:uniform-fix}, we deduce fixation from Lemma \ref{lem:fixation}.
\end{proof} \\ \\
 In view of Lemma~\ref{lem:biased-fix}, the last step to complete the proof of the theorem is to show the positivity of the expected
 value of the weight function when $F > 2 \theta$ and the system starts from the product measure~\eqref{eq:biased} with~$\rho > 0$ small.
 This is done in the next lemma.
\begin{lemma} --
\label{lem:biased-explicit}
 Assume~\eqref{eq:biased}~and~$F > 2 \theta$.
 Then, $E \phi (e) > 0$ for $\rho > 0$ small.
\end{lemma}
\begin{proof}
 To begin with, we observe that
\begin{equation}
\label{eq:biased-explicit-1}
  \begin{array}{l}
    P \,(\zeta_0 (e) = j) \ = \ \sum_{H (u, v) = j} \,P \,(\eta_0 (x) = u \ \hbox{and} \ \eta_0 (x + 1) = v) \end{array}
\end{equation}
 and that, under the assumption~\eqref{eq:biased},
 $$ \begin{array}{l}
      N_j \ := \ \card \,\{(u, v) \in \Gamma^2 : H (u, v) = j \} \vspace*{0pt} \\ \hspace*{50pt} = \ \displaystyle 2^F \card \,\{v \in \Gamma : H (u_-, v) = j \} \ = \ 2^F \,{F \choose j}. \end{array} $$
 Among the pairs with $H (u, v) = F$, exactly two include both $u_-$ and $u_+$ whereas all the other pairs do not include any
 of these two opinions, which, together with~\eqref{eq:biased-explicit-1}, implies
\begin{equation}
\label{eq:biased-explicit-2}
  P \,(\zeta_0 (e) = F) \ = \ 2 \,\rho (u_-)^2 + (2^F - 2) \,\rho^2
\end{equation}
 while among the pairs with $H (u, v) = j < F$, exactly $4 \,(F \ \hbox{choose} \ j)$ include either $u_-$ or $u_+$ whereas all the other pairs do not include any
 of these two opinions, therefore \eqref{eq:biased-explicit-1} implies
\begin{equation}
\label{eq:biased-explicit-3}
  P \,(\zeta_0 (e) = j) \ = \ 4 \,{F \choose j} \,\rho (u_-) \rho + (2^F - 4) {F \choose j} \,\rho^2.
\end{equation}
 Recalling~\eqref{eq:biased}~and~\eqref{eq:biased-weight} and combining~\eqref{eq:biased-explicit-2}--\eqref{eq:biased-explicit-3}, we deduce
 $$ \begin{array}{rcl}
      E \phi (e) & = & \displaystyle \sum_{j = 0}^{\theta} \ (-j) \left(4 \,{F \choose j} \bigg(\frac{1}{2} - (2^{F - 1} - 1) \,\rho \bigg) \,\rho + (2^F - 4) {F \choose j} \,\rho^2 \right) \vspace*{0pt} \\ && \hspace{10pt} + \
                       \displaystyle \sum_{j = \theta + 1}^{F - 1} (j - 2 \theta) \left(4 \,{F \choose j} \bigg(\frac{1}{2} - (2^{F - 1} - 1) \,\rho \bigg) \,\rho + (2^F - 4) \,{F \choose j} \,\rho^2 \right) \vspace*{0pt} \\ && \hspace{10pt} + \
                       \displaystyle (F - 2 \theta) \left(2 \,\bigg(\frac{1}{2} - (2^{F - 1} - 1) \,\rho \bigg)^2 + (2^F - 2) \,\rho^2 \right). \end{array} $$
 In particular, as a function of the parameter $\rho$, the expected weight is a degree two polynomial with constant term~$(1/2)(F - 2 \theta) > 0$
 therefore it is positive for $\rho > 0$ small.
\end{proof} \\ \\
 Fixation under the assumptions of Theorem~\ref{th:biased-fix} directly follows from Lemma~\ref{lem:fixation} and the previous two lemmas.
 To deduce that the one-dimensional system coexists, we use again the argument following Lemma~\ref{lem:threshold-one} at the end of the previous section.



\begin{thebibliography}{100}

\bibitem{adamopoulos_scarlatos_2012}
 Adamopoulos, A. and Scarlatos, S. (2012). Emulation and complementarity in one-dimensional alternatives of the Axelrod model with binary features.
\emph{Complexity} \textbf{17} 43--49.

\bibitem{adelman_1976}
 Adelman, O. (1976). Some use of some ``symmetries'' of some random process.
\emph{Ann. Inst. H. Poincar\'e Sect. B (N.S.)} \textbf{12}, 193--197.

\bibitem{axelrod_1997}
 Axelrod, R. (1997). The dissemination of culture: a model with local convergence and global polarization.
\emph{J. Conflict Resolut.} \textbf{41} 203--226.

\bibitem{bramson_griffeath_1989}
 Bramson, M. and Griffeath, D. (1989). Flux and fixation in cyclic particle systems.
\emph{Ann. Probab.} \textbf{17} 26--45.

\bibitem{clifford_sudbury_1973}
 Clifford, P. and Sudbury, A. (1973). A model for spatial conflict.
\emph{Biometrika} \textbf{60} 581--588.

\bibitem{deffuant_al_2000}
 Deffuant, G., Neau, D., Amblard, F. and Weisbuch, G. (2000). Mixing beliefs among interacting agents.
\emph{Adv. Compl. Sys.} \textbf{3} 87--98.

\bibitem{haggstrom_2012}
 H\"aggstr\"om, O. (2012). A pairwise averaging procedure with application to consensus formation in the Deffuant model.
\emph{Acta Appl. Math.} \textbf{119} 185--201.

\bibitem{haggstrom_hirscher_2014}
 H\"aggstr\"om, O. and Hirscher, T. (2014).
 Further results on the consensus formation in the Deffuant model. Preprint.

\bibitem{harris_1972}
 Harris, T. E. (1972). Nearest neighbor Markov interaction processes on multidimensional lattices.
\emph{Adv. Math.} \textbf{9} 66--89.

\bibitem{holley_liggett_1975}
 Holley, R. A. and Liggett, T. M. (1975). Ergodic theorems for weakly interacting systems and the voter model.
\emph{Ann. Probab.} \textbf{3} 643--663.


\bibitem{lanchier_2012b}
 Lanchier, N. (2012). The critical value of the Deffuant model equals one half.
\emph{ALEA Lat. Am. J. Probab. Math. Stat.} \textbf{9} 383--402.

\bibitem{lanchier_scarlatos_2013}
 Lanchier, N and Scarlatos, S. (2013). Fixation in the one-dimensional Axelrod model. \emph{Ann. Appl. Probab.} \textbf{23}  2538--2559.

\bibitem{lanchier_scarlatos_2014}
 Lanchier, N. and Scarlatos, S. (2014). Fluctuation versus fixation in the one-dimensional constrained voter model.
 Preprint. Available as arXiv:1310.0401.

\bibitem{lanchier_schweinsberg_2012}
 Lanchier, N. and Schweinsberg, J. (2012). Consensus in the two-state Axelrod model.
\emph{Stochastic Process. Appl.} \textbf{122} 3701--3717.

\bibitem{vazquez_krapivsky_redner_2003}
 V\'azquez, F., Krapivsky, P. L. and Redner, S. (2003).   Constrained opinion dynamics: freezing and slow evolution.
\emph{J. Phys. A} \textbf{36}  L61--L68.

\end{thebibliography}
\end{document}